\documentclass[a4paper,leqno,10pt]{article}

\usepackage{epsf,graphicx,epsfig}
\usepackage{amscd}
\usepackage{amsmath,latexsym,amssymb,amsthm}
\usepackage[nospace,noadjust]{cite}
\usepackage{textcomp}
\usepackage{dsfont}
\usepackage{setspace,cite}
\usepackage{lscape,fancyhdr,fancybox}
\usepackage{stmaryrd}
\usepackage[all,cmtip]{xy}
\usepackage{tikz}
\usepackage{cancel}
\usetikzlibrary{shapes,arrows,decorations.markings}
\usepackage{graphicx}

\usepackage{color}
\usepackage{url}
\usepackage{enumerate}
\usepackage[mathscr]{euscript}

  \theoremstyle{plain}

\swapnumbers
    \newtheorem{thm}{Theorem}[section]
    \newtheorem{prop}[thm]{Proposition}
     
   \newtheorem{lemma}[thm]{Lemma}

    \newtheorem{subsec}[thm]{}
\theoremstyle{definition}
    \newtheorem{defn}[thm]{Definition}
        \newtheorem{remark}[thm]{Remark}
    \newtheorem{exam}[thm]{Example}

\theoremstyle{remark}

\title{}
\author{}
\date{}
\usepackage{amssymb}

\usepackage{hyperref}
\hypersetup{
	colorlinks,
	citecolor=blue,
	filecolor=black,
	linkcolor=blue,
	urlcolor=black
}

\begin{document}

\title{Noncommutative differential calculus on (co)homology of hom-associative algebras}

\author{Apurba Das}

\maketitle

\begin{center}
Department of Mathematics and Statistics,\\
Indian Institute of Technology, Kanpur 208016, Uttar Pradesh, India.\\
Email: apurbadas348@gmail.com
\end{center}



\begin{abstract}
A hom-associative algebra is an algebra whose associativity is twisted by an algebra homomorphism.
It was previously shown by the author that the Hochschild cohomology of a hom-associative algebra $A$ carries a Gerstenhaber structure. In this short paper, we show that this Gerstenhaber structure together with certain operations on the Hochschild homology of $A$ makes a noncommutative differential calculus. As an application, we obtain a Batalin-Vilkovisky algebra structure on the Hochschild cohomology of a regular unital symmetric hom-associative algebra.\\
\end{abstract}



\thispagestyle{empty}


\vspace{0.2cm}

\section{Introduction}
The existence of higher structures (such as Gerstenhaber algebras or Batalin-Vilkovisky algebras) on cohomology or homology of a certain algebraic structure was initiated by M. Gerstenhaber in the study of Hochschild cohomology of associative algebras \cite{gers}. Later, a more sophisticated approach of his result was given by Gerstenhaber and Voronov using operad with multiplication in connections with Deligne's conjecture \cite{gers-voro}. 
However, it can only ensure the existence of a Gerstenhaber structure. In differential geometry and noncommutative geometry, one wants a more concrete structure of differential calculus to understand the full account of the picture \cite{nest-tsy,tamar-tsy}. A pair $(\mathcal{A}, \Omega )$ consisting of a Gerstenhaber algebra $\mathcal{A}$ and a graded space $\Omega$ is called a calculus if $\Omega$ carries a module structure over the algebra $\mathcal{A}$ (given by a map $i$) and a module structure over the Lie algebra $\mathcal{A}^{\bullet + 1}$ (given by a map $\mathcal{L}$) and a differential $B : \Omega_\bullet \rightarrow \Omega_{\bullet +1}$ that mixes $\mathcal{L}, i$ and $B$ by the Cartan-Rinehart homotopy formula. For a smooth manifold $M$, the pair $(\mathcal{X}^\bullet (M), \Omega^\bullet (M))$ of multivector fields and differential forms; for an associative algebra $A$, the pair $(H^\bullet (A), H_\bullet (A))$ of Hochschild cohomology and homology yields differential calculus structure.  In \cite{kowal} Kowalzig extends the result of Gerstenhaber and Voronov by introducing a cyclic comp module over an operad (with multiplication). Such a structure induces a simplicial homology on the underlying graded space of the comp module and certain action maps. When passing onto the cohomology and homology, one gets a differential calculus structure.

In this paper, we first construct a new example of differential calculus in the context of hom-associative algebras. A hom-associative algebra is an algebra whose associativity is twisted by a linear homomorphism \cite{makh-sil-0}. Such twisted structures were first appeared in the context of Lie algebras to study $q$-deformations of Witt and Virasoro algebras \cite{hart}. Hom-associative algebras are widely studied in the last 10 years from various points of view. In \cite{amm-ej-makh,makh-sil} Hochschild cohomology and deformations of hom-associative algebras are studied, whereas, homological perspectives are studied in \cite{hasan}. The homotopy theoretic study of hom-associative algebras are considered in \cite{das3}. In \cite{das1} the present author showed that the space of Hochschild cochains of a hom-associative algebra $A$ carries a structure of an operad with a multiplication yields the cohomology a Gerstenhaber algebra (see also \cite{das2}). Here we show that the space of Hochschild chains of $A$ forms a cyclic comp module over the above-mentioned operad. The induced simplicial homology coincides with the Hochschild homology of $A$. Hence, following the result of Kowalzig, we obtain a differential calculus structure on the pair of Hochschild cohomology and homology of $A$. See Section \ref{section-3} for details clarification.

Finally, as an application, we obtain a Batalin-Vilkovisky algebra structure on the Hochschild cohomology of a regular unital symmetric hom-associative algebra. This generalizes the corresponding result for associative algebras obtained by Tradler \cite{tradler}.

Throughout the paper, $k$ is a commutative ring of characteristic $0$. All linear maps and tensor products are over $k$.

\section{Noncommutative differential calculus and cyclic comp modules}
In this section, we recall noncommutative differential calculus and cyclic comp modules over non-symmetric operads. We mention how a cyclic comp module induces a noncommutative differential calculus. Our main references are \cite{gers-voro,kowal}. We mainly follow the sign conventions of \cite{kowal}.

\begin{defn}
(i) A Gerstenhaber algebra over $k$ is a graded $k$-module $\mathcal{A} = \oplus_{i \in \mathbb{Z}} \mathcal{A}^i$ together with a graded commutative, associative product $\cup : \mathcal{A}^p \otimes \mathcal{A}^q \rightarrow \mathcal{A}^{p+q}$ and a degree $-1$ graded Lie bracket $[~, ~]: \mathcal{A}^p \otimes \mathcal{A}^q \rightarrow \mathcal{A}^{p+q-1}$ satisfying the following Leibniz rule
\begin{align*}
[f, g \cup h ] = [f , g ] \cup h + (-1)^{(p-1) q} g \cup [f, h], ~~ \text{ for } f \in \mathcal{A}^p, g \in \mathcal{A}^q, h \in \mathcal{A}.
\end{align*}

(ii) A pair $(\mathcal{A}, \Omega)$ consisting of a Gerstenhaber algebra $\mathcal{A} = ( \mathcal{A}, \cup, [~, ~])$ and a graded $k$-module $\Omega$ is called a precalculus if there is a graded $(\mathcal{A}, \cup)$-module structure on $\Omega$ given by $i : \mathcal{A}^p \otimes \Omega_n \rightarrow \Omega_{n-p}$ and a graded Lie algebra module by $\mathcal{L} : \mathcal{A}^{p+1} \otimes  \Omega_n \rightarrow \Omega_{n-p}$ satisfying
\begin{align}\label{eqn-t}
i_{[f, g]} = i_f \circ \mathcal{L}_g - (-)^{p (q+1)} \mathcal{L}_g \circ i_f, ~ \text{ for } f \in \mathcal{A}^p, g \in \mathcal{A}^q.
\end{align}

(iii) A precalculus $(\mathcal{A}, \Omega)$ is said to be a calculus if there is a degree $+1$ map $B : \Omega_\bullet \rightarrow \Omega_{\bullet + 1}$ satisfying $B^2 =0$ and the following Cartan-Rinehart homotopy formula holds
\begin{align*}
\mathcal{L}_f = B \circ i_f - (-1)^p~ i_f \circ B, ~ \text{ for } f \in \mathcal{A}^p.
\end{align*}
\end{defn}

\begin{defn}
A non-symmetric operad $\mathcal{O}$ in the category of $k$-modules consists of a collection $\{ \mathcal{O}(p) \}_{p \geq 1}$ of $k$-modules together with $k$-bilinear maps (called partial compositions) $\circ_i : \mathcal{O}(p) \otimes \mathcal{O}(q) \rightarrow \mathcal{O}(p+q-1)$, for $1 \leq i \leq p$, satisfying the following identities
\begin{align*}
(f \circ_i g) \circ_j h = \begin{cases} (f \circ_j h) \circ_{i+p-1} g ~~~&\mbox{if } j <i\\
f \circ_i (g \circ_{j-i+1} h) ~~~&\mbox{if } i \leq j < q+i\\ 
(f \circ_{j-q+1} h) \circ_i g ~~~&\mbox{if } j \geq q+i \end{cases}
\end{align*}
for $f \in \mathcal{O}(p),~ g \in \mathcal{O}(q),~ h \in \mathcal{O}(r)$;
and there is a distinguished element $\mathds{1} \in \mathcal{O}(1)$ satisfying $\mathds{1} \circ_1 f = f = f \circ_i \mathds{1}$, for $f \in \mathcal{O}(p)$ and $1 \leq i \leq p$.
\end{defn}

In an operad $\mathcal{O}$, there is a degree $-1$ bracket $[~, ~]: \mathcal{O}(p) \otimes \mathcal{O}(q) \rightarrow \mathcal{O}(p+q-1)$ given by
\begin{align*}
[f, g] := \sum_{i=1}^p (-1)^{(q-1)(i-1)} f \circ_i g - (-1)^{(p-1)(q-1)} \sum_{i=1}^q (-1)^{(p-1)(i-1)} g \circ_i f.
\end{align*}

Let $\mathcal{O} = (\mathcal{O}, \circ_i, \mathds{1})$ be a non-symmetric operad. An element $\pi \in \mathcal{O}(2) $ is called a multiplication on $\mathcal{O}$ if it satisfies $\pi \circ_1 \pi = \pi \circ_2 \pi$. Note that a multiplication induces a differential $\delta_\pi : \mathcal{O}(n) \rightarrow \mathcal{O}(n+1),~ \delta_\pi (f) = [\pi, f].$
We denote the corresponding cohomology groups by $H^n_\pi (\mathcal{O})$, for $n \geq 1$. 
Further, a multiplication $\pi$ induces a degree $0$ product $f \cup g := (\pi \circ_2 f) \circ_1 g$. In \cite{gers-voro} the authors showed that these two operations passes onto the cohomology $H^\bullet_\pi (\mathcal{O})$ and make it a Gerstenhaber algebra.

\begin{defn}
(i) A comp module over an operad $\mathcal{O} = (\mathcal{O}, \circ_i, \mathds{1}) $ consists of a sequence of $k$-modules $\{ \mathcal{M} (n) \}_{n \geq 0}$ together with $k$-bilinear operations (called comp module maps) $\bullet_i : \mathcal{O} (p) \otimes \mathcal{M}(n) \rightarrow \mathcal{M}(n-p+1)$, for $p \leq n$ and $1 \leq i \leq n-p+1$, satisfying the following identities
\begin{align}\label{eqn-p}
f \bullet_i ( g \bullet_j x ) = \begin{cases} g \bullet_j ( f \bullet_{i+q-1} x ) ~~~ & \mbox{ if } j < i \\
( f \circ_{j-i+1} g) \bullet_i x ~~~ & \mbox{ if } j-p \leq i \leq j \\
g \bullet_{j-p+1} ( f \bullet_i x) ~~~ & \mbox{ if } 1 \leq i \leq j-p, \end{cases}
\end{align}
for $f \in \mathcal{O}(p), g \in \mathcal{O}(q)$ and $x \in \mathcal{M}(n)$. It is called unital comp module if
\begin{align}\label{eqn-q}
\mathds{1} \bullet_i x = x, ~ \text{ for } i =1, \ldots, n.
\end{align}

(ii) A cyclic (unital) comp module over $\mathcal{O}$ is a unital comp module $\mathcal{M}$ equipped with an additional comp module map $\bullet_0 : \mathcal{O}(p) \otimes \mathcal{M}(n) \rightarrow \mathcal{M} (n-p+1)$, for $p \leq n+1$ such that the relations (\ref{eqn-p}) and (\ref{eqn-q}) hold $i=0$ as well; and a $k$-linear map $t : \mathcal{M}(n) \rightarrow \mathcal{M}(n)$, for $n \geq 1$, satisfying $t^{n+1} = \mathrm{id}$ and
\begin{align}\label{eqn-r}
t ( f \bullet_i x) = f \bullet_{i+1} t(x), ~ \text{ for } i= 0, 1, \ldots, n-p.
\end{align}
\end{defn}

Let $(\mathcal{O}, \pi)$ be a non-symmetric operad with a multiplication and $\mathcal{M}$ be a cyclic unital comp module over $\mathcal{O}$. Then there is a simplicial boundary map $b : \mathcal{M}(n) \rightarrow \mathcal{M}(n-1)$, for $n \geq 1$, given by
\begin{align}\label{simplicial-boun}
b (x) = \sum_{i=0}^{n-1} (-1)^i~ \pi \bullet_i x + (-1)^n ~\pi \bullet_0 t (x).
\end{align}
We denote the corresponding homology groups by $H_\bullet (\mathcal{M}).$ Moreover, there is a cap product $i_f := f \cap \_ : \mathcal{M}(n) \rightarrow \mathcal{M}(n-p)$, for $f \in \mathcal{O}(p)$ given by
\begin{align}\label{cap-pro}
i_f x = ( \pi \circ_2 f) \bullet_0 x, ~\text{ for } x \in \mathcal{M}(n)
\end{align}
and a Lie derivative $\mathcal{L}_f : \mathcal{M}(n) \rightarrow \mathcal{M}(n-p+1)$, for $f \in \mathcal{O}(p)$ given by
\begin{align}\label{lie-pro}
\mathcal{L}_f x := \begin{cases}  \sum_{i=1}^{n-p+1} (-1)^{(p-1)(i-1)} f \bullet_i x + \sum_{i=1}^p (-1)^{n (i-1) + p-1} f \bullet_0 t^{i-1} (x) ~~ & \mbox{ if } p < n+1\\
(-1)^{p-1} \sum_{i=0}^n (-1)^{in} f \bullet_0 t^i (x)  ~~ & \mbox{ if } p = n+1. \end{cases}
\end{align}
It has been shown in \cite[Proposition 4.3, Theorem 4.4]{kowal} that these two operators satisfy the following identities
\begin{align}
&i_{f \cup g} = i_f \circ i_g, ~~~~ i_{\delta_\pi f} = b \circ i_f - (-1)^p~ i_f \circ b, \label{eqn-ab}\\
&\mathcal{L}_{[f, g]} = \mathcal{L}_f \circ \mathcal{L}_g - (-1)^{(p-1)(q-1)} \mathcal{L}_g \circ \mathcal{L}_f , ~~~~ \mathcal{L}_{\delta_\pi f} = - b \circ \mathcal{L}_f + (-1)^{(p-1)} \mathcal{L}_f \circ b. \label{eqn-cd}
\end{align}
It follows from the above identities that the cap product and the Lie derivative descend to the simplicial homology $H_\bullet (\mathcal{M})$ and the graded homology $H_\bullet (\mathcal{M})$ is a module over both the algebra $(H^\bullet_\pi (\mathcal{O}), \cup)$ and the graded Lie algebra $(H^{\bullet + 1}_\pi (\mathcal{O}), [~, ~])$. It has been further proved in \cite[Theorem 4.5]{kowal} that the induced operations on (co)homology satisfies the identity (\ref{eqn-t}) to make the pair $(H^\bullet_\pi (\mathcal{O}), H_\bullet (\mathcal{M))}$ is a precalculus.

Some additional structure on the operad $\mathcal{O}$ makes the above precalculus into a calculus. We start with the following.
\begin{defn}\label{defn-unit-opera} An operad with multiplication $(\mathcal{O}, \pi)$ is called unital if there is a $k$-module $\mathcal{O}(0)$ so that the partial compositions $\circ_i$ extend to $\mathcal{O}(0)$, and there is an element $e \in \mathcal{O}(0)$ satisfying $\pi \circ_1 e = \pi \circ_2 e = \mathds{1}$.
\end{defn}

If $(\mathcal{O}, \pi)$ is unital and $\{ \mathcal{M}(n) \}_{n \geq 0}$ is a cyclic comp module over $\mathcal{O}$ (in the sense that (\ref{eqn-p}) and (\ref{eqn-r}) holds for $p=0$ as well), then one may define a differential $B : \mathcal{M}(n) \rightarrow \mathcal{M}(n+1)$, called the Connes boundary operator, by
\begin{align*}
B := \sum_{i=0}^n (-1)^{in}  (\mathrm{id} -t ) t ( e \bullet_{n+1} t^i (x)).
\end{align*}
Like classical Hochschild case, one may consider normalized cochain complex $\{ \overline{\mathcal{O}}, \delta_\pi \}$ that induces the same cohomology with $H^\bullet_\pi (\mathcal{O})$. Similarly, a normalized chain complex $\{ \overline{\mathcal {M}}, b \}$ can be considered which induces the same homology with $H_\bullet (\mathcal{M})$ \cite{kowal}.
It has been further shown in \cite{kowal} that on the normalized (co)chain complexes
\begin{align}\label{cartan-pre-id}
\mathcal{L}_f = [B, i_f ] + [b, S_f] - S_{\delta_\pi f}, \text{ for } f \in \overline{\mathcal{O}}(p),
\end{align}
where the map $S_f : \mathcal{M} (n) \rightarrow \mathcal{M} (n-p+2)$ is given by
\begin{align*}
S_f = \sum_{j=1}^{n-p+1} \sum_{i=j}^{n-p+1} (-1)^{n(j-1) + (p-1)(i-1) } ~e \bullet_0 ( f \bullet_i t^{j-1} (x)), \text{ for } 0 \leq p \leq n.
\end{align*}
and $S_f = 0$, for $p > n$. Thus, it follows from (\ref{cartan-pre-id}) that the Cartan-Rinehart homotopy formula holds on the induced (co)homology. Hence the pair $(H^\bullet_\pi (\mathcal{O}), H_\bullet (\mathcal{M}))$ is a differential calculus.

\section{Hom-associative algebras and calculus structure}\label{section-3}
In this section, we first recall hom-associative algebras and their Hochschild (co)homologies \cite{makh-sil-0}, \cite{amm-ej-makh}, \cite{hasan}. In the next, we show that the pair of cohomology and homology forms a precalculus. Under some additional conditions on the hom-associative algebra, the precalculus turns out to be a  noncommutative differential calculus.

\begin{defn}
A hom-associative algebra is a $k$-module $A$ together with a $k$-bilinear map $\mu : A \otimes A \rightarrow A, (a, b) \mapsto a \cdot b $ and a $k$-linear map $\alpha : A \rightarrow A$ satisfying the following hom-associativity:
\begin{align*}
(a \cdot b ) \cdot \alpha (c) = \alpha (a) \cdot ( b \cdot c ), ~ \text{ for } a, b, c \in A.
\end{align*} 
\end{defn}
A hom-associative algebra as above is denoted by the triple $(A, \mu, \alpha )$. It is called multiplicative if $\alpha ( a \cdot b ) = \alpha (a) \cdot \alpha (b)$, for $a, b \in A$. In the rest of the paper, by a hom-associative algebra, we shall always mean a multiplicative hom-associative algebra.
It follows from the above definition that any associative algebra is a (multiplicative) hom-associative algebra with $\alpha = \mathrm{id}_A$.

A hom-associative algebra $(A, \mu, \alpha )$ is said to be unital if there is an element $1 \in A$ such that $\alpha (1) = 1$ and $a \cdot 1 = 1 \cdot a = \alpha (a),$ for all $a \in A$.

\begin{exam}\label{hom-alg-exam}
Let $(A, \mu)$ be an associative algebra over $k$ and $\alpha : A \rightarrow A$ be an algebra homomorphism. Then $(A, \mu_\alpha =\alpha \circ \mu , \alpha )$ is a hom-associative algebra, called obtained by composition. If the associative algebra $(A, \mu)$ is unital and $\alpha$ is an unital associative algebra morphism then the hom-associative algebra $(A, \mu_\alpha =\alpha \circ \mu , \alpha )$ is unital with the same unit.
\end{exam}

Let $(A, \mu, \alpha)$ be a hom-associative algebra. A bimodule over it consists of a $k$-module $M$ and a linear map $\beta : M \rightarrow M$ with actions $l : A \otimes M \rightarrow M,~ (a,m) \mapsto am$, and $r: M \otimes A \rightarrow M, ~ (m,a) \mapsto ma$ satisfying $\beta (am) = \alpha (a) \beta(m),~ \beta (ma) = \beta(m) \alpha (a)$ and
the following bimodule conditions are hold
\begin{align*}
(a \cdot b)\beta( m) = \alpha(a) (bm) ~~~~ (am)\alpha(b) = \alpha(a)(mb) ~~~~ (ma) \alpha(b) = \beta(m) (a \cdot b), ~ \text{ for } a, b \in A, m \in M.
\end{align*} 
A bimodule can be simply denoted by $(M, \beta)$ when the actions are understood. It is easy to see that $(A, \alpha)$ is a bimodule with left and right actions are given by $\mu$. 

Let $(A, \mu, \alpha)$ be a hom-associative algebra and $(M, \beta)$ be a bimodule over it. The group of $n$-cochains of $A$ with coefficients in $(M, \beta)$ is given by $C^n_\alpha ( A, M) := \{ f : A^{\otimes n} \rightarrow M |~ \beta \circ f = f \circ \alpha^{\otimes n} \},$ for $n \geq 1$. The coboundary map $\delta_\alpha : C^n_\alpha (A, M) \rightarrow C^{n+1}_\alpha (A, M)$ given by
\begin{align}
(\delta_\alpha f)(a_1, \ldots, a_{n+1}) :=~& \alpha^{n-1}(a_1)  f ( a_2, \ldots, a_{n+1})  + (-1)^{n+1} f (a_1, \ldots, a_n) \alpha^{n-1} (a_{n+1}) \\
~&+ \sum_{i=1}^n (-1)^i~ f ( \alpha (a_1), \ldots, \alpha (a_{i-1}), a_i \cdot a_{i+1}, \alpha (a_{i+2}), \ldots, \alpha (a_{n+1})). \nonumber
\end{align}
The corresponding cohomology groups are denoted by $H^n_\alpha (A, M)$, for $n \geq 1$. When the bimodule is given by $(A, \alpha)$, the corresponding cochain groups are denoted by $C^n_\alpha (A)$ and the cohomology groups are denoted by $H^n_\alpha (A)$, for $n \geq 1$.

In this paper, we only require the Hochschild homology of $A$ with coefficients in itself. For homology with coefficients, see \cite{hasan}. The $n$-th Hochschild chain group of $A$ with coefficients in itself is given by $C_n^\alpha (A) := A \otimes A^{\otimes n}$, for $n \geq 0$ and the boundary operator $d^\alpha : C_n^\alpha (A) \rightarrow C_{n-1}^\alpha (A)$ given by
\begin{align}\label{hoch-hom-boundary}
d^\alpha ( a_0 \otimes a_1 \cdots a_n) :=~& a_0 \cdot a_1 \otimes \alpha (a_2) \cdots \alpha (a_n) + (-1)^n a_n \cdot a_0 \otimes \alpha (a_1) \cdots \alpha (a_{n-1}) \\
~&+ \sum_{i=1}^{n-1} (-1)^i~ \alpha (a_0) \otimes \alpha (a_1) \cdots (a_i \cdot a_{i+1}) \cdots \alpha (a_n). \nonumber
\end{align}
The corresponding homology groups are denoted by $H_n^\alpha (A)$, for $n \geq 0$.


Let $(A, \mu, \alpha )$ be a hom-associative algebra. It has been shown in \cite{das1} that the collection of Hochschild cochains $\mathcal{O} = \{ \mathcal{O}(p) = C^p_\alpha (A) \}_{p \geq 1}$ forms a non-symmetric operad with partial compositions
\begin{align}\label{hom-ope}
(f \circ_i g ) ( a_1, \ldots, a_{p+q-1}) = f ( \alpha^{q-1} (a_1), \ldots, \alpha^{q-1} ( a_{i-1}), g (a_i, \ldots, a_{i+q-1}), \ldots, \alpha^{q-1} ( a_{p+q-1})),
\end{align}
for $f \in \mathcal{O}(p), g \in \mathcal{O}(q)$ and the identity element $\mathds{1} = \mathrm{id}_A.$ Moreover, the element $\mu \in \mathcal{O}(2) = C^2_\alpha (A)$ is a multiplication in the above operad. The differential induced by $\mu$ is same as $\delta_\alpha$ up to a sign. Hence the graded space of Hochschild cohomology $H_\alpha^\bullet (A)$ carries a Gerstenhaber structure.

\subsection{Precalculus structure}

Let $\mathcal{M}(n) := A \otimes A^{\otimes n}$, for $n \geq 0$ be the $n$-th Hochschild chain group of $A$. For convenience, we denote the basic elements of $\mathcal{M}(n)$ by $a_0 \otimes a_1 \cdots a_n$ when there is no confusion causes. For $p \leq n$ and $1 \leq i \leq n-p+1$, we define maps $\bullet_i : \mathcal{O}(p) \otimes \mathcal{M}(n) \rightarrow \mathcal{M}(n-p+1)$ by
\begin{align*}
f \bullet_i ( a_0 \otimes a_1 \cdots a_n ) := \alpha^{p-1}( a_0 ) \otimes \alpha^{p-1}( a_1) \cdots \alpha^{p-1}(a_{i-1}) f (a_i, \ldots, a_{i+p-1}) \alpha^{p-1}(a_{i+p}) \cdots \alpha^{p-1}( a_n). 
\end{align*} 
\begin{prop}
With these notations, $\mathcal{M} = \{ \mathcal{M}(n) \}_{n \geq 0}$ is a unital comp module over the operad $\mathcal{O}$.
\end{prop}

\begin{proof}
We have to verify the identities (\ref{eqn-p}) and (\ref{eqn-q}). First, for $j < i$, we have
\begin{align*}
&f \bullet_i ( g \bullet_j (a_0 \otimes a_1 \cdots a_n )) \\&= f \bullet_i \big( \alpha^{q-1}(a_0) \otimes \alpha^{q-1}(a_1) \cdots g (a_j, \ldots, a_{j+q-1}) \cdots \alpha^{q-1} (a_n) \big) \\
&= \alpha^{p+q-2} (a_0) \otimes \alpha^{p+q-2} (a_1) \cdots \alpha^{p+q-2} (a_{j-1}) \alpha^{p-1} ( g (a_j, \ldots, a_{j+q-1} )) \cdots \\& \qquad \qquad \qquad \qquad \qquad \qquad f (\alpha^{q-1} (a_{i+q-1}), \ldots, \alpha^{q-1}(a_{i+p+q-2})) \cdots \alpha^{p+q-2} (a_n).
\end{align*}
On the other hand,
\begin{align*}
&g \bullet_j ( f \bullet_{i+q-1} (a_0 \otimes a_1 \cdots a_n )) \\
&= g \bullet_j \big( \alpha^{p-1} (a_0) \otimes \alpha^{p-1} (a_1) \cdots f (a_{i+q-1} , \ldots, a_{i+p+q-2}) \cdots \alpha^{p-1} (a_n)    \big)\\
&= \alpha^{p+q-2} (a_0) \otimes \alpha^{p+q-2} (a_1) \cdots \alpha^{p+q-2} (a_{j-1})  g ( \alpha^{p-1}(a_j), \ldots, \alpha^{p-1}(a_{j+q-1}) ) \cdots \\& \qquad \qquad \qquad \qquad \qquad \qquad \alpha^{q-1}(f ( a_{i+q-1}, \ldots, a_{i+p+q-2})) \cdots \alpha^{p+q-2} (a_n).
\end{align*}
Hence $f \bullet_i ( g \bullet_j (a_0 \otimes a_1 \cdots a_n )) = g \bullet_j ( f \bullet_{i+q-1} (a_0 \otimes a_1 \cdots a_n ))$. Similarly, for  $j-p <i \leq j$, we can verify that
\begin{align*}
&f \bullet_i ( g \bullet_j (a_0 \otimes a_1 \cdots a_n )) \\
&= \alpha^{p+q-2} (a_0) \otimes \alpha^{p+q-2} (a_1) \cdots f \big(\alpha^{q-1} (a_i), \ldots, g (a_j, \ldots, a_{j+q-1}), \ldots, \alpha^{q-1} (a_{i+p+q-2})    \big) \cdots \alpha^{p+q-2} (a_n)\\
&= (f \circ_{j-i+1} g) \bullet_i (a_0 \otimes a_1 \cdots a_n ).
\end{align*}
Finally, for $1 \leq i \leq j-p$, we have
\begin{align*}
&f \bullet_i ( g \bullet_j (a_0 \otimes a_1 \cdots a_n )) \\ 
&= \alpha^{p+q-2} (a_0) \otimes \alpha^{p+q-2} (a_1) \cdots f (\alpha^{q-1}(a_i), \ldots, \alpha^{q-1}(a_{i+p-1})) \cdots \alpha^{p-1} (g (a_j , \ldots, a_{j+q-1})) \cdots \alpha^{p+q-2} (a_n)\\
&= g \bullet_{j-p+1} ( f \bullet_i (a_0 \otimes a_1 \cdots a_n )).
\end{align*}
It is also easy to see that $\mathds{1} \bullet_i (a_0 \otimes a_1 \cdots a_n ) = (a_0 \otimes a_1 \cdots a_n )$, for all $a_0 \otimes a_1 \cdots a_n \in \mathcal{M}(n)$. Hence the proof.
\end{proof}

We define additional maps $\bullet_0 : \mathcal{O}(p) \otimes \mathcal{M}(n) \rightarrow \mathcal{M}(n-p+1)$, for $p \leq n+1$ and $t : \mathcal{M} (n) \rightarrow \mathcal{M}(n)$ by
\begin{align*}
f \bullet_0 ( a_0 \otimes a_1 \cdots a_n ) :=~& f(a_0, \ldots, a_{p-1}) \otimes \alpha^{p-1}( a_p) \cdots \alpha^{p-1}( a_n),\\
t (  a_0 \otimes a_1 \cdots a_n ) :=~& a_n \otimes a_0 a_1 \cdots a_{n-1}.
\end{align*}

\begin{prop}
With these additional structures, $\mathcal{M}$ is a cyclic comp module over $\mathcal{O}$.
\end{prop}

\begin{proof}
Similar to the proof of the previous proposition, one shows that $\bullet_0$ satisfy the identities (\ref{eqn-p}) and (\ref{eqn-q}). The operator $t$ obviously satisfies $t^{n+1} = \mathrm{id}$. Finally, for any $1 \leq i \leq n-p$,
\begin{align*}
t ( f \bullet_i ( a_0 \otimes a_1 \cdots a_n ) ) =~& t ( a_0 \otimes a_1 \cdots f (a_i, \ldots, a_{i+p-1}) \cdots a_n )\\
=~& a_n \otimes a_0 \cdots f (a_i, \ldots, a_{i+p-1})  \cdots a_{n-1} \\
=~& f \bullet_{i+1} (a_n \otimes a_0 a_1 \cdots a_{n-1}) = f \bullet_{i+1} t (a_0 \otimes a_1 \cdots a_n ).
\end{align*}
The same holds for $i=0$. Hence the proof.
\end{proof}

Consider the operad $\mathcal{O}$ with the multiplication $\pi \in \mathcal{O}(2)$.  As $\mathcal{M}$ is a cyclic comp module over $\mathcal{O}$, the corresponding simplicial boundary map $b $ as of (\ref{simplicial-boun}) is precisely given by the Hochschild boundary operator $d^\alpha$ given in (\ref{hoch-hom-boundary}). Hence $H_\bullet (\mathcal{M})$ is given by the Hochschild homology $H_\bullet^\alpha (A)$ of the hom-associative algebra $A$. 

Note that, in this example, the corresponding cap product (\ref{cap-pro}) and Lie derivative (\ref{lie-pro}) are given by
\begin{center}
 $i_f ( a_0 \otimes a_1 \cdots a_n ) =  \alpha^{p-1} (a_0) \cdot f (a_1, \ldots, a_p) \otimes \alpha^p (a_{p+1}) \cdots \alpha^p (a_n),$
 \end{center}
 \begin{align*}
 \mathcal{L}_f ( a_0 \otimes& a_1 \cdots a_n ) = \sum_{i=1}^{n-p+1} (-1)^{(p-1)(i-1)} ~\alpha^{p-1} (a_0) \otimes \alpha^{p-1} (a_1) \cdots f (a_i, \ldots, a_{i+p-1}) \cdots \alpha^{p-1} (a_n)\\
 &+ (-1)^{p-1} f (a_0, \ldots, a_{p-1}) \otimes \alpha^{p-1} (a_p) \cdots \alpha^{p-1} (a_n) \\
& + \sum_{i=2}^{p-1} (-1)^{n(i-1) + p-1} ~f(a_{n-i+2}, \ldots, a_{n-i+p+1}) \otimes \cdots \alpha^{p-1} (a_n) \alpha^{p-1}(a_0) \cdots \alpha^{p-1} (a_{n-i+1})\\
& + (-1)^{(n+1)(p-1)}~ f(a_{n-p+2}, \ldots, a_n, a_0) \otimes \alpha^{p-1}(a_1) \cdots \alpha^{p-1} (a_{n-p+1}), ~ \text{ for } p < n+1.
\end{align*}

Hence we get the following.

\begin{thm}
Let $(A, \mu, \alpha )$ be a hom-associative algebra. Then the pair $(H^\bullet_\alpha (A), H_\bullet^\alpha (A))$ of Hochschild cohomology and Hochschild homology of $A$ forms a precalculus.
\end{thm}

\subsection{Calculus structure}

A hom-associative algebra $(A, \mu, \alpha)$ is said to be regular if $\alpha$ is invertible. In such a case, it can be easily checked that $(A, \mu_{\alpha^{-1}} = \alpha^{-1} \circ \mu)$ is an associative algebra and $\alpha$ is an algebra morphism. Moreover, the hom-associative algebra $(A, \mu, \alpha)$ is then obtained by composition (Example \ref{hom-alg-exam}).

Let $(A, \mu, \alpha)$ be a regular unital hom-associative algebra with unit $1 \in A$. Then the corresponding operad with multiplication $(\mathcal{O}, \mu)$ considered in (\ref{hom-ope}) is unital in the sense of Definition \ref{defn-unit-opera} with $\mathcal{O}(0) = \{ a \in A | \alpha (a) = a \}$ and the partial compositions $\circ_i : \mathcal{O}(p) \otimes \mathcal{O}(q) \rightarrow \mathcal{O}(p+q-1)$, for $p, q \geq 0$ given by the same formula (\ref{hom-ope}). The element $1 \in \mathcal{O}(0)$ is unit as
\begin{align*}
(\mu \circ_1 1 )(a) = \mu (1, \alpha^{-1}(a)) = 1 \cdot \alpha^{-1} (a) = a ~~~ \text{~ and ~} ~~~
(\mu \circ_2 1 )(a) = \mu (\alpha^{-1}(a), 1) = \alpha^{-1} (a) \cdot 1 = a.
\end{align*}

\begin{remark}
In this case, the Hochschild cochain complex is defined from degree $0$ with $C^0_\alpha (A) = \mathcal{O}(A)$ and $\delta_\alpha : C^0_\alpha (A) \rightarrow C^1_\alpha (A)$ by $\delta_\alpha (a) (b) =  \alpha^{-1}(b) \cdot a - a \cdot \alpha^{-1}(b)$, for $a \in C^0_\alpha (A)$ and $b \in A$. More generally, if a hom-associative algebra $A$ is regular and a bimodule $(M, \beta)$ is also regular (i.e. $\beta$ is invertible) then the Hochschild cohomology of $A$ with coefficients in $(M, \beta)$ can be defined from degree $0$.
\end{remark}

It is also easy to see that the Hochschild chains $\{ \mathcal{M} (n) \}_{n \geq 0}$ is infact a cyclic comp module over the unital operad $\mathcal{O}$ in the sense that the identities (\ref{eqn-p}) and (\ref{eqn-r}) holds for $p= 0$. 
Hence by the result of the previous section, we deduce the following.
\begin{thm}
Let $(A, \mu, \alpha)$ be a regular unital hom-associative algebra. Then the pair  $(H^\bullet_\alpha (A), H_\bullet^\alpha (A))$ is a noncommutative differential calculus.
\end{thm}

\section{Application}

Using the calculus of the previous section and some additional hypothesis, we construct a Batalin-Vilkovisky algebra structure on the Hochschild cohomology of a hom-associative algebra. In particular, the Hochschild cohomology of a regular unital symmetric hom-associative algebra carries a Batalin-Vilkovisky algebra structure. We start with the following definition.

\begin{defn}
A Gerstenhaber algebra $(\mathcal{A}, \cup, [~, ~])$ is said to be a Batalin-Vilkovisky algebra (BV algebra in short) if there is a BV-generator, i.e. a degree $-1$ map $\triangle : \mathcal{A}^\bullet \rightarrow \mathcal{A}^{\bullet -1}$ satisfying $\triangle^2 = 0 $ and
\begin{align*}
[f, g ] = - (-1)^p  \big( \triangle ( f \cup g) -  \triangle (f) \cup g - (-1)^p f \cup \triangle (g)   \big), ~ \text{ for } f \in \mathcal{A}^p, ~ g \in \mathcal{A}.
\end{align*}
\end{defn}

\begin{lemma}\label{fin-lem}
In a calculus $(\mathcal{A}, \Omega )$, we have for $f \in \mathcal{A}^p$ and $g \in \mathcal{A}^q$,
\begin{align*}
i_{[f, g]} x = (-1)^{q+1}~ i_{f \cup g} B (x) + (-1)^{p+1} B (i_{f \cup g} x) + i_f B ( i_g x ) + (-1)^{pq+p+q}~ i_g  B ( i_f x). 
\end{align*}
\end{lemma}

\begin{proof}
We have
\begin{align*}
i_{[f, g]} =~& [ i_f, \mathcal{L}_g ] = i_f \circ \mathcal{L}_g - (-1)^{ p (q+1)} \mathcal{L}_g \circ i_f \\
=~& i_f \circ ( B \circ i_g - (-1)^q ~i_g \circ B) -(-1)^{p (q+1)} ( B \circ i_g - (-1)^q ~i_g \circ B) \circ i_f  \\
=~& i_f \circ B \circ i_g + (-1)^{q+1}~ i_{f \cup g} \circ B  - (-1)^{p(q+1)} (-1)^{pq} B \circ i_{f \cup g}  + (-1)^{p (q+1) + q} i_g \circ B \circ i_f.
\end{align*}
Hence the proof.
\end{proof}

\begin{prop}\label{final-prop}
Let $(A, \mu, \alpha)$ be a regular unital hom-associative algebra. If there is an element $c \in H^\alpha_d (A)$ such that $B (c) = 0$ and the maps $H^n_\alpha (A) \rightarrow H^\alpha_{d-n} (A),~ f \mapsto i_f c$ are $H^\bullet_\alpha (A)$-module isomorphisms, then the map 
\begin{align*}
\triangle : H^\bullet_\alpha (A) \rightarrow H^{\bullet -1}_\alpha (A) ~~\text{ defined by } ~~ i_{\triangle (f)} c = B (i_f c)
\end{align*}
is a BV-generator on the Gerstenhaber algebra $H^\bullet_\alpha (A)$. In other words, $H^\bullet_\alpha (A)$ is a BV algebra.
\end{prop}

\begin{proof}
We have from Lemma \ref{fin-lem} that
\begin{align*}
i_{[f, g]} c =~& - (-1)^ p ~ B \circ i_{f \cup g} (c) + i_f \circ B \circ i_g (c) + (-1)^{pq+p+q} ~ i_g \circ B \circ  i_f (c) \\
=~& -(-1)^p i_{\triangle ( f \cup g)} c + i_f \circ i_{\triangle (g)} c + (-1)^{pq + p + q} i_g i_{\triangle (f)} c \\
=~& - (-1)^p \big(   i_{\triangle ( f \cup g)} c - (-1)^p ~ i_{f \cup \triangle (g)} c - i_{\triangle (f) \cup g} c \big).
\end{align*}
Hence the result follows from the given hypothesis.
\end{proof}

In the next, we give a dual analogue of Lemma \ref{fin-lem}, hence a dual version of Proposition \ref{final-prop}. Before that, we need the followings.

Let $(A, \mu, \alpha)$ be a regular unital hom-associative algebra. Then it can be easily checked that the dual space $A^*$ with the linear map $(\alpha^{-1})^* : A^* \rightarrow A^*$ is a bimodule over the hom-associative algebra $A$ with left and right actions given by $(a \theta )(b)  =  \theta ( \alpha^{-1} (b \cdot a))$ and $(\theta a ) (b) = \theta ( \alpha^{-1} (a \cdot b))$, for $a, b \in A$ and $\theta \in A^*$.

Let $B: C^\alpha_n (A, A) \rightarrow C^\alpha_{n+1}(A,A)$ be the Connes boundary map. Then the dual map $B^* : \mathrm{Hom}(A^{\otimes n+1} , A^*) \rightarrow  \mathrm{Hom}(A^{\otimes n} , A^*) $ restricts to a map (denoted by the same notation)  $B^* : C^{n+1}_\alpha (A, A^*) \rightarrow C^n_\alpha (A, A^*)$, for $n \geq 0$ that also passes onto the cohomology $H^\bullet_\alpha (A, A^*)$. The induced map on cohomology is also denoted by $B^*$.

For $f \in C^p_\alpha (A)$ and $m \in C^{|m|}_\alpha (A, A^*)$, we define a product $f \cdot m \in C^{p+|m|}_\alpha (A, A^*)$ by
\begin{align*}
(f \cdot m)(a_1 , \ldots, a_{p+|m|}) = f(a_1, \ldots, a_p ) m (a_{p+1}, \ldots, a_{p+|m|})
\end{align*}
that induces a product on the cohomology level.

\begin{lemma}
Let $(A, \mu, \alpha)$ be a regular unital hom-associative algebra. For $f \in H^p_\alpha (A), ~ g \in H^q_\alpha (A)$ and $m \in H^\bullet_\alpha (A, A^*)$, we have
\begin{align*}
[f, g] \cdot m = (-1)^{p+1} B^* ( ( f \cup g) \cdot m) + f \cdot B^* ( g \cdot m) + (-1)^{pq+p+q}~ g \cdot B^* ( f \cdot m) + (-1)^{q+1} ( f \cup g) \cdot B^* (m).
\end{align*}
\end{lemma}

\begin{proof}
We have from the definition of $B^*$ that $B^*(m) = (-1)^{|m|} m \circ B$ and from the Koszul sign convension
\begin{align*}
m ( i_f  x) = (-1)^{|m| p } (f \cdot m )(x).
\end{align*}
On the other hand, from Lemma \ref{fin-lem}, we have
\begin{align*}
m (i_{[f,g]} x) = (-1)^{q+1}~ m (i_{f \cup g} B(x)) + (-1)^{p+1}~ m \circ B (i_{f \cup g} x) + m ( i_f B (i_g x)) + (-1)^{pq+p+q}~ m (i_g B (i_f x)).
\end{align*}
Hence we get
\begin{align*}
(-1)^{|m|(p+q-1)} ([f,g] \cdot m)(x) =~&  (-1)^{q+1 + |m| (p+q) + p+q + |m|}~ B^* ((f \cup g) \cdot m)(x) \\~&+ (-1)^{p+1+|m| + (|m|+1) (p+q)}~ ((f \cup g) \cdot B^*(m)) (x) \\
~&+ (-1)^{|m| p + (p+|m| +1) q + p + |m|}~ (g \cdot B^*(f \cdot m))(x) \\
~&+ (-1)^{pq+p+q + |m|q + (q + |m| +1)p + q + |m|}~ (f \cdot B^*(g \cdot m ))(x).
\end{align*}
Thus the result follows by cancelling the sign $(-1)^{|m|(p+q-1)}$ from both sides.
\end{proof}

\begin{prop}\label{prop-last-prop}
Let $(A, \mu, \alpha)$ be a regular unital hom-associative algebra. If there is an element $m \in H^d_\alpha (A, A^*)$ such that $B^* (m) = 0$ and the maps $H^\bullet_\alpha (A) \rightarrow H^{\bullet + d}_\alpha (A, A^*), ~ f \mapsto f \cdot m$ are $H^\bullet_\alpha (A)$-module isomorphisms, then
\begin{align*}
\triangle : H^\bullet_\alpha (A) \rightarrow H^{\bullet -1}_\alpha (A, A) ~~\text{ defined by }~~ (\triangle f) \cdot m = B^* ( f \cdot m)
\end{align*}
makes the Gerstenhaber algebra $H^\bullet_\alpha (A)$ into a BV algebra.
\end{prop}

\begin{proof}
We have from the previous lemma that
\begin{align*}
[f, g] \cdot m = (-1)^{p+1} \triangle ( f \cup g) \cdot m  + ( f \cup \triangle (g)) \cdot m + (-1)^{pq + p + q} (g \cup \triangle (f)) \cdot m.
\end{align*}
The result follows by applying the graded commutativity on the last term of the right-hand side and from the given hypothesis.
\end{proof}

A regular unital hom-associative algebra $(A, \mu, \alpha)$ is said to be symmetric if there is a hom-bimodule isomorphism $\Theta : A \xrightarrow{\sim} A^*.$ Thus, in a symmetric algebra, there is an isomorphism $H^\bullet_\alpha (A) \xrightarrow{\sim} H^\bullet_\alpha (A, A^*)$ via the map $\Theta$. Hence the dual map $B^* : H^\bullet_\alpha (A, A^*) \rightarrow H^{\bullet -1}_\alpha (A, A^*)$ induces a map (denoted by the same notation) $B^* : H^*_\alpha (A) \rightarrow H^{\bullet -1}_\alpha (A).$

\begin{thm}
Let $(A, \mu, \alpha)$ be a regular unital symmetric hom-associative algebra. Then the map $B^* : H^*_\alpha (A) \rightarrow H^{\bullet -1}_\alpha (A)$ defines a BV algebra structure on the Gerstenhaber algebra $H^\bullet_\alpha (A).$
\end{thm}

\begin{proof}
Consider the element $\Theta (1) \in A^* = C^0_\alpha (A, A^*)$ which is a $0$-cocycle. Hence $m = [ \Theta (1) ] \in H^0_\alpha (A, A^*)$. Since $B^*$ decreases degree, we have $B^* (m) = 0$. Moreover, the maps $H^\bullet_\alpha (A) \rightarrow H^\bullet_\alpha (A, A^*),~ f \mapsto f \cdot m \simeq f$ are isomorphisms. With these isomorphisms, the map $\triangle$ given in Proposition \ref{prop-last-prop} coincides with $B^*$, as
\begin{align*}
\triangle (f) \cong ( \triangle f) \cdot m = B^* ( f \cdot m ) \cong B^* (f).
\end{align*}
Hence the result follows from Proposition \ref{prop-last-prop}.
\end{proof}
\begin{remark}
This generalizes the corresponding result for associative algebras \cite{tradler}. However, our proof relies on the noncommutative differential calculus associated with (hom-)associative algebras.
\end{remark}
\medskip

\noindent {\bf Acknowledgements.} The research is supported by the fellowship of Indian Institute of Technology (IIT) Kanpur. The author thanks the Institute for support.

\medskip

\noindent {\bf Data Availability Statement.} Data sharing is not applicable to this article as no new data were created or analyzed in this study.

\end{document}